\documentclass[reqno, 12pt]{amsart}
\usepackage[utf8]{inputenc}
\usepackage{amssymb,amsmath,amsfonts,amsthm,calrsfs,mathtools}
\usepackage{color}
\usepackage[english]{babel}
\usepackage[T1]{fontenc}
\usepackage{latexsym}
\usepackage{enumitem}
\usepackage{tabularx}
\usepackage[a4paper,top=3cm,bottom=2cm,left=3cm,right=3cm,marginparwidth=1.75cm]{geometry}
\usepackage[colorinlistoftodos]{todonotes}
\usepackage[colorlinks=true, allcolors=red]{hyperref}
\theoremstyle{plain}
\newtheorem{theorem}{Theorem}
\newtheorem{lemma}{Lemma}

\newtheorem{proposition}{Proposition}

\newtheorem{corollary}{Corollary}

\theoremstyle{definition}

\newtheorem{definition}{Definition}

\newtheorem{remark}{Remark}

\providecommand{\keywords}[1]

\title{Dynamical Properties for Composition Operators on \texorpdfstring{$H^{2}(\mathbb{C}_{+})$}{lg}}
\author{Carlos F. \'{A}lvarez}

\address{
Instituto de Matemáticas, Pontificia Universidad Católica de Valparaíso, Blanco Viel 596, Cerro Barón, Valparaíso, Chile}
\email{
carlos.alvarez.e@pucv.cl}

\author{Javier Henríquez-Amador}
\address{Departamento de Matemáticas, Universidad del Valle, Ciudadela Universitaria Meléndez
Edificio E20, Cali, Colombia}
\email{javier.henriquez@correounivalle.edu.co}

\thanks{CF. A. was supported by ANID Proyecto FONDECYT Postdoctorado 3230048, Chile}

\makeatletter

\date{\today}

\begin{document}

\begin{abstract}
In this article, we investigate expansivity, Li-Yorke chaos and shadowing properties for composition operators $C_{\phi}f = f \circ \phi$ induced by affine self-maps $\phi$ of the right half-plane $\mathbb{C}_{+}$ on the Hardy-Hilbert space $H^{2}(\mathbb{C_{+}})$.
\end{abstract}

\maketitle
\markright{Dynamical Properties for Composition Operators on Hardy Space}


\section{Introduction}
In this paper, $\mathbb{C}$ denotes the complex plane and $\mathbb{C}_{+}:=\{z\in \mathbb{C}: Re(z)>0\}$ represents the open right half-plane. $\mathbb{D}$ and $\mathbb{T}$ denote the open unit disc and the unit circle in $\mathbb{C}$, respectively. Let $U\subset \mathbb{C}$ be an open set and let $Y$ be a Banach space of holomorphic functions $\phi:U \to U$. 
A \textit{composition operator with symbol $\phi$} is defined by $$C_{\phi}f=f\circ \phi, \  {\rm{for\ any \ }} f\in Y.$$ 

Composition operators are a class of linear operators studied in functional analysis, complex dynamics, and ergodic theory. When $U=\mathbb{C}_+$ and $C_{\phi}$ is defined on the Hardy Hilbert space of the open right half-plane $H^{2}(\mathbb{C}_{+})$, this operator has been widely studied in \cite{EJ,FavaroHaiSeveriano, KumarSri,M1,M2, NS}. It is important to mention that the only linear fractional self-maps of $\mathbb{C}_+$ that induce continuous composition operators on $H^2(\mathbb{C}_+)$ are the \textit{affine maps} 
\begin{center}
$\phi(w)=aw+b$ with $a>0$ and $Re(b)\geq 0.$
\end{center}\vspace{0.2cm}

\textit{Linear Dynamics} field focuses on the study of dynamical systems that involve continuous linear operators on infinite-dimensional Banach (or Fréchet) spaces. Within this area, important concepts such as \textit{chaos, expansivity} and \textit{shadowing} have emerged. For precise definitions of these notions see Subsection \ref{sld}. Exploring these dynamical properties has become a common interest among researchers in the fields of functional analysis and dynamical systems. For a more detailed understanding, readers are encouraged to refer to the books \cite{BM,GP} and the papers \cite{BBMP2, ES}, which provide additional references. To delve into the dynamics of linear operators on Hilbert spaces, readers can consult \cite{GMM2017, MM, P}.\vspace{0.2cm}

In the early 1970’s Eisenberg and Hedlund studied relationships between expansivity and spectrum of operators on Banach spaces \cite{EH}. The \textit{Li-Yorke chaos} notion was introduced by Li and Yorke in \cite{LY} for interval maps in 1975. This version of chaos has also been studied within context of linear dynamics, for instance, as exemplified \cite{BBMAP, BBMP2} and related references. It is worth mentioning that Godefroy and Shapiro introduced chaos for linear operators on Fréchet spaces \cite{GS}, adopting Devaney's definition of chaos, see \cite{GP}. Starting the 2000's, Mazur investigated expansivity and shadowing for the class of normal operators on Hilbert spaces \cite{MM}. In 2018, the notions of expansivity and shadowing has been very good explored for continuous linear operators on Banach spaces \cite{ES}.\vspace{0.2cm}

There have been several works exploring composition operators within the framework of linear dynamics. For instance, Li-Yorke chaos and expansivity has been characterized for composition operators $T_{\phi}: f\mapsto f \circ \phi$ acting on $L^{p}$-spaces with $1\leq  p < \infty$, see \cite{BDP} and \cite{MExP}, respectively. Other kind of dynamical properties for weighted composition operators acting on the smooth function space were investigated in \cite{P17}.\vspace{0.2cm}

In this paper, we explore expansivity, shadowing and Li-Yorke chaos for the composition operator $C_{\phi}$ defined on $H^{2}(\mathbb{C}_+)$. These results was inspired by the results of Noor and Severiano in \cite{NS}, since the class of composition operators under investigation is not hypercyclic, that is, $(C^{n}_{\phi}f)_{n\in \mathbb{N}\cup \{0\}}$ is not dense in $H^{2}(\mathbb{C}_+)$. Moreover, the notions of expansivity and shadowing has not been explored in this context. We provide examples of operators exhibiting expansivity and the shadowing property which are not hyperbolic, but they are not Li-Yorke chaotic. These examples show the richness of linear dynamics and how this differs from finite nonlinear dimensional dynamics.\vspace{0.2cm}

We aim to contribute to dynamical properties of composition operators $C_{\phi}$ on the Hardy Hilbert space $H^{2}(\mathbb{C}_+)$. A particular emphasis is placed on the shadowing property, a key concept that connects the local and global dynamics of operators. The shadowing property, originally studied in the context of nonlinear dynamics, provides a framework for understanding the approximation of pseudo-orbits by true orbits in the operator setting, see \cite{Kawaguchi}. We characterize the symbols $\phi$ that induce operators $C_{\phi}$ with some kind of shadowing. \vspace{0.2cm}

Let us now describe the organization of the article. In section \ref{preli} we fix the notations and recall the definitions and a few results which will be important in our work. In section \ref{EH2}, we characterize some notions of expansivity and show that in certain cases, there exist some equivalences between these notions. In section \ref{SH2}, we provide conditions to guarantee or not the shadowing property and we give a characterization of the positive shadowing property. Finally, we show that composition operators on the Hardy space of the open right half-plane can not be Li-Yorke chaotic. 

\section{Preliminaries}\label{preli}
In this section, we fix some terminology and recall some definitions and results about spectral theory, linear dynamics and Hardy space on the open right half-plane.  Let $X$ be a complex Banach space, the set $S_X = \{x \in X : \|x\|= 1\}$ is the \textit{unit sphere} of $X$. In addition, by an operator on $X$ we mean a continuous linear map $T:X \to X$. The \textit{spectrum} $\sigma(T)$ of $T$ is the set $$\sigma(T)= \{\lambda \in \mathbb{C} : T-\lambda I \ {\rm{is \ not \ invertible \ continuous \ operator}} \}.$$

Recall that $\sigma(T)\subset \mathbb{C}$ is always a non-empty compact set. Moreover, if $T$ is a self-adjoint operator on a complex Hilbert space $H$ and $x\in H$, then there is a unique positive Radon measure $\mu$ on $\sigma(T)$ such that $$\langle f(T)x, x\rangle
 =\displaystyle\int_{\sigma(T)}f(t)d\mu(t) \ \ {\rm{for \ all}} \ f \in C(\sigma(T)).$$
So, $\mu(\sigma(T)) = \|x\|^{2}$. The measure $\mu$ is called the \textit{spectral measure} associated to $T$ and $x$.

\subsection{Notions of linear dynamics}\label{sld}
Expansivity is a well known property of dynamical systems. For continuous operators, we provide a definition that contain several versions of expansivity. 
\begin{definition}\label{d1}
Let $T:X\to X$ be an operator on $X$. We say that
\begin{enumerate}
    \item $T$ is \textit{expansive} if $T$ is invertible and for every $z\in S_{X}$ there exists $n\in \mathbb{Z}$ such that $\|T^{n}z\|\geq 2.$
    \item  $T$ is \textit{positively expansive} if for every $z\in S_{X}$ there exists $n\in \mathbb{N}$ such that $\|T^{n}z\|\geq 2.$
    \item  $T$ is \textit{uniformly expansive} if $T$ is invertible and there exists $n\in \mathbb{N}$ such that $$z\in S_{X}\Longrightarrow \|T^{n}z\|\geq 2 \ {\rm{or}} \ \|T^{-n}z\|\geq 2.$$
    \item  $T$ is \textit{uniformly positively expansive} if there exists $n\in \mathbb{N}$ such that $$z\in S_{X}\Longrightarrow \|T^{n}z\|\geq 2.$$
\end{enumerate}
\end{definition}

The following result characterizes the expansivity for operators defined on Hilbert spaces. 
\begin{theorem}[{\cite[Theorem~23]{ES}}]\label{es}
Let $T$ be a continuous operator on a Hilbert space $H$. The following assertions hold:
\begin{enumerate}
    \item $T$ is positively expansive if and only if  $\displaystyle\sup_{n\in\mathbb{N}}\left\|T^{n}(x)\right\|=+\infty,$ for every nonzero $x\in H$.
    \item $T$ is uniformly positively expansive if and only if $\displaystyle\lim_{n\rightarrow+\infty}\left\|T^{n}(x)\right\|=+\infty$ uniformly on $S_{H}$.
    If, in addition, $T$ is an invertible normal operator then 
    \item $T$ is expansive if and only if $\sigma_{p}(T^{*}T)\cap\mathbb{T}=\emptyset.$
 \end{enumerate}
\end{theorem}
 
The definition of shadowing has a simplified formulation in the setting of linear dynamics. Here, we use this formulation. 

\begin{definition}
Let $T:X \to X$ be an operator on $X$. A sequence $(x_n)_{n\in \mathbb{Z}}\subset X$ is called a $\delta$-pseudotrajectory of $T$, where $\delta > 0$, if
$$\|T x_{n} - x_{n+1}\| \leq \delta, \  {\rm{for \ all}}  \ n\in \mathbb{Z}.$$
\end{definition}

\begin{definition}
 Let $T : X\to X$ be an invertible operator on $X$. Then $T$ is said to have the \textit{shadowing property} if for every $\epsilon>0$ there exists $\delta > 0$ such that every $\delta$-pseudotrajectory $(x_n)_{n\in \mathbb{Z}}$ of $T$ is $\epsilon$-shadowed by a real trajectory of $T$, that is, there exists $x\in X$ such that  $$\|T^{n}x - x_{n}\| \leq \epsilon, \  {\rm{for \ all}} \ n\in \mathbb{Z}.$$
\end{definition}

When $T$ is an operator not necessarily invertible, we can define the notion of positive shadowing property for $T$. In such case, we replace the set $\mathbb{Z}$ by $\mathbb{N}$ in the above definition.

\begin{theorem}\label{spr}
Let $T:X\to X$ be an invertible operator on $X$. Suppose that $X =M \oplus N$, where $M$ and $N$ are closed subspaces of $X$ with $T(M)\subset M $ and $T^{-1}(N) \subset N$. If $\sigma(T|M)\subset \mathbb{D}$ and $\sigma(T^{-1}|N) \subset \mathbb{D}$, then $T$ has the shadowing property. 
\end{theorem}

An operator $T:X\to X$ is said to be \textit{hyperbolic} if $\sigma(T) \cap \mathbb{T} = \emptyset$. As consequence of Theorem \ref{spr}, we have the following results:

\begin{corollary}\label{corsh}
Every invertible hyperbolic operator $T:X\to X$ has the shadowing property. 
\end{corollary}

\begin{corollary}\label{corsh1}
If $T$ is an invertible normal operator on a Hilbert space $H$, then $T$ has the shadowing property
if and only if $T$ is hyperbolic.    
\end{corollary}

Finally, Let us also recall that an operator $T$ on a Banach space $X$ is Li-Yorke chaotic if it has an uncountable \textit{scrambled set} $S$, i.e., for all $x, y \in S$ with $x\neq y$, we have that $$\displaystyle\liminf_{n\to \infty}\|T^{n}x-T^{n}y\|=0\ and 
 \ \displaystyle\limsup_{n\to \infty}\|T^{n}x-T^{n}y\|=\infty.$$

The following result is well-known with respect to the recent advances of linear dynamics.

\begin{theorem}[{\cite[Theorem~9]{BBMP2}}]\label{mainLY}
Let $T$ be an operator on $X$. The following assertions are equivalent: 
\begin{enumerate}
    \item $T$ is Li-Yorke chaotic; 
    \item $T$ admits a {\rm{semi-irregular vector}}, that is, a vector $y\in X$ such that $$ \displaystyle\liminf_{n\to \infty}\|T^{n}y\|=0\ and \ \displaystyle\limsup_{n\to \infty}\|T^{n}y\|>0;$$
    \item $T$ admits an {\rm{irregular vector}}, that is, a vector $z\in X$ such that $$\displaystyle\liminf_{n\to \infty}\|T^{n}z\|=0\ and 
 \ \displaystyle\limsup_{n\to \infty}\|T^{n}z\|=\infty.$$
\end{enumerate}
\end{theorem}

An important consequence of the previous result is the following corollary.

\begin{corollary}[{\cite[Corollary~6]{BBMAP}}]\label{mainCor}
Let $T$ be a Li-Yorke chaotic operator. The following assertions hold:
\begin{enumerate}
    \item $\sigma(T)\cap \mathbb{T}\neq \emptyset$.
    \item $T$ is not normal.
    \item $T$ is not compact.
\end{enumerate}
\end{corollary}

\subsection{Hardy space \texorpdfstring{$H^2(\mathbb{C}_{+})$}{lg}}
Let $\mathbb{C}_{+}$ be the open right half-plane. The Hardy space  $H^{2}(\mathbb{C}_{+})$ is the Hilbert space of analytic functions on $\mathbb{C}_{+}$ for which the norm $$\|f\|_2^2=\displaystyle\sup_{0<x<\infty}\frac{1}{\pi}\displaystyle\int_{-\infty}^{\infty}|f(x+iy)|^2dy$$ is finite. For each $\beta\in \mathbb{C}_{+}$, let $k_{\beta}$ denote the \textit{reproducing kernel} for $H^2(\mathbb{C}_{+})$ at $\beta$, that is, $$k_{\beta}(w)=\frac{1}{w+\overline{\beta}}.$$

These kernels satisfy the fundamental relation $\langle f, k_{\beta} \rangle =f(\beta)$ for all $f\in H^2(\mathbb{C}_{+})$.
If $\phi:\mathbb{C}_{+}\to \mathbb{C}_{+}$ is a holomorphic map and $C_{\phi}$ is the composition operator with symbol $\phi$, then $C^{\ast}_{\phi}k_{\beta}=k_{\phi(\beta)},$ for every $\beta \in \mathbb{C}_{+}.$
\vspace{0.2cm}

Elliott and Jury \cite[Theorem~3.1]{EJ} proved that a holomorphic map $\phi:\mathbb{C}_+ \to \mathbb{C}_+$ induces a continuous operator $C_{\phi}$ on $H^{2}(\mathbb{C}_+)$ if and only if $\phi(\infty)=\infty$ and if the non-tangential limit \begin{equation*}\label{Eq1}
    \phi^{'}(\infty):=\displaystyle\lim_{w\to \infty}\frac{w}{\phi(w)}
\end{equation*}
exists and is finite, in which case, $\|C_{\phi}\| = \sqrt{\phi^{'}(\infty)}$. Matache \cite{M2} proved that the only linear fractional self-maps of $\mathbb{C}_+$ that induce continuous composition operators on $H^2(\mathbb{C}_+)$ are the \textit{affine maps} 
\begin{equation}\label{am}
\phi(w)=aw+b,
\end{equation}
where $a>0$ and $Re(b)\geq 0.$ Such a map $\phi$ is said to be of \textit{parabolic-type} if $a = 1$ and is a \textit{parabolic automorphism} if additionally $Re(b) = 0$. Similarly, $\phi$ is of hyperbolic-type if $a \neq 1$ and is a \textit{hyperbolic automorphism} if additionally $Re(b) = 0$. A hyperbolic non-automorphisms are precisely the symbols $\phi(w) = aw + b$ with $a \in (0, 1)\cup (1,\infty)$ and $Re(b)>0$.

The next result shows that if $\phi$ is as in (\refeq{am}), then the normal composition operators $C_{\phi}$ on $H^{2}(\mathbb{C}_+)$ are completely characterized.

\begin{theorem}[{\cite[Theorem~2]{NS}}]\label{normal}
Let $\phi(w) = aw + b$ with $a > 0$ and $Re(b)\geq 0$. Then \begin{enumerate}
\item $C_{\phi}$ is normal if and only if $a = 1$ or $Re(b) = 0$,
\item $C_{\phi}$ is self-adjoint if and only if $a = 1$ and $b\geq 0$,
\item $C_{\phi}$ is unitary if and only if $a = 1$ and $Re(b) = 0.$ 
\end{enumerate}
\end{theorem}

For simplicity, we call as \textit{hyperbolic of type I} to every  map $\phi(w)=aw+b$ with $a\in (0,1)$ and $Re(b)\geq 0.$  Similarly, we call as \textit{hyperbolic of type II} to every  map $\phi(w)=aw+b$ with $a\in (1,\infty)$ and $Re(b)\geq 0.$ For each $n\in \mathbb{N}$, let $\phi^{[n]}$ denote the $n$-th iterate of the symbol $\phi(w) = aw + b$, then a simple computation gives
 \begin{equation}\label{eq3}
 \phi^{[n]}(w)=\left\{\begin{array}{lcc} w+nb, &  & a=1, \\
 a^{n}w+\displaystyle\frac{(1-a^n)}{1-a}b, &  & a\neq1.
 \end{array}
 \right.
 \end{equation}

We finish this section recalling known results about the spectrum of the operator $C_{\phi}$ on $H^2(\mathbb{C}_+)$ (see Theorems 1.1 and 1.2 of \cite{S}).

\begin{theorem}[\cite{S}]\label{sp1}
Let $\phi(w)=w+b$ be a parabolic self-map of $\mathbb{C}_+$ with $Re(b)\geq 0$. 
Then the spectrum of $C_{\phi}$ acting on $H^2(\mathbb{C}_+)$ is
\begin{enumerate}
    \item $\sigma(C_{\phi})= \mathbb{T}$, when $Re(b)=0$,
    \item $\sigma(C_{\phi} )= \{e^{bt} : t\in [0, \infty)\}\cup \{0\},$ when $b\in \mathbb{C}_+.$ 
\end{enumerate}
\end{theorem}

\begin{theorem}[\cite{S}]\label{sp2}
Let $\phi(w)=aw+b$ be a hyperbolic self-map of $\mathbb{C}_+$ with $Re(b)\geq 0$ and $b\neq 0$. Then the spectrum of $C_{\phi}$ acting on $H^{2}(\mathbb{C}_+)$ is
\begin{enumerate}
    \item $\sigma(C_{\phi})= \{\lambda \in \mathbb{C}: |\lambda| = a^{-1/2}\},$ when $Re(b)=0$,
  \item $\sigma(C_{\phi})= \{\lambda\in \mathbb{C}: |\lambda|\leq a^{-1/2} \}$, when $b\in \mathbb{C}_+$.
\end{enumerate}

\end{theorem}

\section{Expansivity property for  \texorpdfstring{$C_{\phi}$}{lg}}\label{EH2}
In this section, we fully characterize the different notions of expansivity for the family of all composition operators on $H^2(\mathbb{C}_{+})$ induced by a certain class of affine maps. In section 3.1, we begin by studying the positive uniform expansivity and the positive expansivity showing as a main result that they are equivalent including also the affine map type. In section 3.2, we study the case when the composition operator is invertible and show that uniform expansivity and uniform expansivity are equivalent properties. 

\subsection{Expansivity in a positive sense of \texorpdfstring{$C_{\phi}$}{lg}} 
To explore the (uniformly) positive expansivity of $C_{\phi}$ on $H^2(\mathbb{C}_{+})$, the main complication arises when the affine symbol is hyperbolic of type I since the resolvent of the operator $C_{\phi}$ does not intersect the unit disk and in this order the result in \cite[Theorem~D]{ES} is not applicable. Fortunately, using the density of kernels we can overcome this difficulty. 

\begin{theorem}\label{exp} Let $\phi(w)=aw+b$ be a self-map of $\mathbb{C}_{+}$. Then, the following statements are equivalent:
\begin{enumerate}
\item $C_{\phi}$ is positively expansive on $H^{2}(\mathbb{C}_{+})$;
\item $\phi$ is \textit{hyperbolic of type I};
\item $C_{\phi}$ is uniformly positively expansive on $H^{2}(\mathbb{C}_{+}).$
\end{enumerate}
\end{theorem}

In order of studying the uniform expansiveness, some properties of the ball $S_{H^{2}(\mathbb{C}_+)}$ need to be studied. 

\begin{lemma}\label{l1} Let $\phi$ be hyperbolic of type $I.$ Then, for each non-zero $f\in H^2(\mathbb{C}_{+})$ it follows $f(\gamma)\neq0$ where $\gamma:=b/1-a$ is the fixed point of $\phi.$ 
\end{lemma}

\begin{proof}
    Suppose $f(\gamma)=0$ for some non-zero $f\in H^2(\mathbb{C}_{+}).$ Then, there exist $R_0>0$ such that $f(z)\neq0$ for $0<|z-\gamma|<R_0.$ By using \eqref{eq3}, $$\phi^{[n]}(z)-\gamma=a^{n}(z-\gamma), \ \ \textrm{for} \ \ z\in\mathbb{C}_{+}.$$ Hence, since $a\in(0,1)$ it follows $\phi^{[n]}(B_R)\subset B_R$ for any ball $B_R$ of radius $R$ in $\mathbb{C}_{+}.$ For each $k\in\mathbb{N}$ define $G_k:=\{z\in B_{R_0/2^k}\setminus \{\gamma\}: Re(z-\gamma)>0\},$ since $f\neq0$ on $G_k$ an application of the Minimal Modulus Principle implies that there exist $z_k\in\partial G_k$ such that $|f(z)|\geq |f(z_k)|$ for all $z\in G_k.$ In particular, $$|f(\phi^{[n]}(z))|\geq |f(z_k)|, \ \ \textrm{for} \ \ n=1,2,3,\cdots \ \ \textrm{and} \ \ z\in G_k.$$ 

    Since $f(\phi^{n}(z))\to f(\gamma)=0, \ \ \textrm{uniformly in } \  z\in\overline{G_k}$ it follows that $f(z_k)=0$ for each $k\in\mathbb{N}.$ However, since $z_k\in G_k$ 
    
    $$\displaystyle\sum_{k=1}^{+\infty} \frac{Re(z_k)}{1+|z_k|^{2}}\geq Re(\gamma)\sum_{k=1}^{+\infty}\frac{2^{2k}}{(2^k|\gamma|+R_0)^{2}}=\sum_{k=1}^{+\infty}\frac{1}{(|\gamma|+R_0/2^k)^{2}}=+\infty,$$
    thus the zero's sequence $\{z_k\}$ does not satisfies the zero conditions of \textit{Baschke} for $f\in H^{2}(\mathbb{C}_+)$ unless $f=0.$     
\end{proof}

\begin{proof} (Proof the Theorem \ref{exp})
Recall that $\phi$ is \textit{hyperbolic of type I}, when $a\in(0,1)$ and $Re(b)\geq 0.$
$(1)\Longrightarrow(2)$ 
    If $\|C^{n}_{\phi}\|\leq 1$ for all $n\in\mathbb{N}$, then 
$\sup\{\|C^{n}_{\phi}f\| : n\in \mathbb{N}\}$ is finite, however this contradicts $C_{\phi}$ to be positively expansive. Therefore, there exists $n_0$ such that $\sqrt{\frac{1}{a^{n_0}}}=\|C^{n_{0}}_{\phi}\|>1$ and hence $a<1.$  
$(2)\Longrightarrow (3)$ Let $\gamma:=b/1-a$ be fixed point of $\phi.$ It is not difficult see that $S_{H^2(\mathbb{C}_+)}$ is bounded locally in $H^2(\mathbb{C}_+),$ thus $S_{H^2(\mathbb{C}_+)}$ is equi-continuous in each $z\in\mathbb{C}_+$. Hence, since $\phi^n(z)\to \gamma$ as $n\to+\infty
$ punctually in $z\in \mathbb{C}_+$ this implies $f(\phi^n(z))\to f(\gamma)$ as $n\to+\infty
$  punctually in $z\in\mathbb{C}_+$ and uniformly in $f\in S_{H^2(\mathbb{C}_+)}.$ By Fatou Lemma, for each $x\in(0,+\infty)$ fixed 
\begin{eqnarray*}
\liminf_{n\to+\infty}
\|C_\phi^{n}f\|&\geq& \frac{1}{\pi} \liminf_{n\to+\infty}\int^{+\infty}_{-\infty} |f(\phi^{n}(x+iy))|^2 \ dy\\
&\geq& \frac{1}{\pi} \liminf_{n\to+\infty}\int_{-\frac{\pi}{f^2(\gamma)}}^{\frac{\pi}{f^2(\gamma)}} \ |f(\phi^{n}(x+iy))|^2 \ dy \\
&\geq& \frac{1}{\pi} \int_{-\frac{\pi}{f^2(\gamma)}}^{\frac{\pi}{f^2(\gamma)}} \ \liminf_{n\to+\infty}|f(\phi^{n}(x+iy))|^2 \ dy \\
&=& 2
\end{eqnarray*}
that is, $\liminf
\|C_\phi^{n}f\|\geq 2$ uniformly in $f\in S_{H^2(\mathbb{C}_+)}.$ Therefore, $C_{\phi}$ is uniformly positively expansive on $H^{2}(\mathbb{C}_{+}).$ Finally, $(3)\Longrightarrow (1)$ is obvious.
\end{proof}

\subsection{Forward and backward expansivity for \texorpdfstring{$C_{\phi}$}{lg}} We characterize the expansivity for invertible composition operators on $H^2(\mathbb{C}_{+}).$

\begin{theorem}
Let $\phi(w)=aw+ib$ with $b\in\mathbb{R}$. Then, the following statements are equivalent:
\begin{enumerate}
\item $C_{\phi}$ is expansive on $H^{2}(\mathbb{C}_{+})$;
\item $C_{\phi}$ is uniformly expansive on $H^{2} (\mathbb{C}_{+})$;
\item $\phi$ is hyperbolic-type. 
\end{enumerate}
\end{theorem} 
\begin{proof} We first show that $\sigma(C_{\phi}^{*}C_{\phi})=\{a^{-1}\}.$ Indeed, since $Re(ib)=0,$ it follows that $\phi$ induces an invertible and normal bounded composition operator on $H^{2}(\mathbb{C}_{+})$ (see \cite[Theorem~2.4]{M2}). Note that $C_{\phi}^{-1}=    C_{\phi^{-1}}$ and $\phi^{-1}(w)=a^{-1}w-a^{-1}ib.$ In this case, the adjoint formula gives $C_{\phi}^{*}=a^{-1}C_{\phi^{-1}}.$ Then $$\sigma(C_{\phi}^{*}C_{\phi})=\sigma(a^{-1}C_{\phi^{-1}}C_{\phi})=a^{-1}\sigma(I)=\{a^{-1}\}.$$ 
$(1)\Longleftrightarrow(3).$ By Theorem \ref{es} part (3), $C_{\phi}$ is expansive if and only if $a\neq 1.$ This implies that (1) and (3) are equivalent. \

$(2)\Longleftrightarrow(3)$ Suppose that $\phi$ is hyperbolic-type, then $a\in(0,1)\cup(1,+\infty).$ We first consider $a\in(0,1).$ For $n$ such that $\sqrt{1/a^{n}}\geq 2$ and $f\in S_{H^{2}(\mathbb{C})},$ we have $$\|C_{\phi}^{n}f\|\geq \frac{1}{\|C_{\phi^{-1}}^{n}\|}=\sqrt{1/a^{n}}\geq 2.$$ 
Similarly, if $a\in(1,+\infty)$ then there exists $n\in\mathbb{N}$ such that $\sqrt{a^n}\geq 2.$ Hence, for each $f\in S_{H^{2}(\mathbb{C})}$ we have $\|(C_{\phi}^{-1})^{n}f\|\geq \sqrt{a^n}\geq 2.$ 
\end{proof}

\section{Shadowing and Li-Yorke chaos for \texorpdfstring{$C_{\phi}$}{lg}}\label{SH2}
In this section, we characterize positive shadowing for continuous composition operators on $H^2(\mathbb{C}_{+})$ induced for a certain class of affine maps. 

\begin{proposition}\label{ksha}
Let $\phi:\mathbb{C}_{+}\rightarrow \mathbb{C}_{+}$ be a holomorphic map such that $C_{\phi}$ is bounded on $H^2(\mathbb{C}_{+})$. If $\phi$ has a fixed point in $\mathbb{C}_+$ then $C_{\phi}$ does not have the positive shadowing property on $H^2(\mathbb{C}_{+})$. 
\end{proposition}
\begin{proof}
Let $\eta$ be a fixed point of $\phi$ in $\mathbb{C}_{+}$. For $\delta>0$, consider $(f_{n}^{\delta})_{n\in\mathbb{N}}\subset H^2(\mathbb{C}_{+})$ given by $$f_{n}^{\delta}:=\displaystyle\frac{\delta}{\|C_{\phi}f\|}\sum_{k=0}^{n-1}C_{\phi}^{n-k}f, \ \ n\in \mathbb{N},$$
where $f\in H^2(\mathbb{C}_{+})$ is chosen such that $f(\eta)\neq0$. Then,                                 
$(f_{n}^{\delta})_{n\in\mathbb{N}}$ is a
$\delta$-pseudotrajectory of $C_{\phi}$ for all $\delta>0$. In fact, by standard arguments we obtain
\begin{eqnarray*}    \|C_{\phi}f_{n}^{\delta}-f_{n+1}^{\delta}\|&=&\displaystyle\frac{\delta}{\|C_{\phi}f\|}\left\|\left[\sum_{k=0}^{n-1}C_{\phi}^{n-k+1}f-\sum_{k=0}^{n}C_{\phi}^{n-k+1}f\right]\right\| \\    &=&\displaystyle\frac{\delta}{\|C_{\phi}f\|}\left\|\left[\sum_{k=0}^{n-1}C_{\phi}^{n-k+1}f-\sum_{k=0}^{n-1}C_{\phi}^{n-k+1}f-C_{\phi}f\right]\right\|\\
&=&\frac{\delta}{\|C_{\phi}f\|} \|C_{\phi}f\| \\
&=\delta.&
\end{eqnarray*}
Now, notice that $$f_{n}^{\delta}(\eta)=\displaystyle\frac{n\delta}{\|C_{\phi}f\|}f(\eta), \ \  n\in \mathbb{N}.$$ Consequently,
for any $g\in H^2(\mathbb{C}_{+})$, we have 
$$\|C_{\phi}^{n}g-f_{n}^{\delta}\|\geq \sqrt{2 Re(\eta)} \ |g(\phi^{n}(\eta))-f_{n}^{\delta}(\eta)| \geq\sqrt{2 Re(\eta)} \ \left|\frac{n\delta}{\|C_{\phi}f\|}|f(\eta)|-|g(\eta)|\right|$$
that is, $$\|C_{\phi}^{n}g-f_{n}^{\delta}\|\geq \sqrt{2Re(\eta)} \left|\frac{n\delta}{\|C_{\phi}f\|}|f(\eta)|-|g(\eta)|\right|, \ \ n\in \mathbb{N}.$$
Hence, $\|C_{\phi}^{n}g-f_{n}^{\delta}\|\rightarrow+\infty$ as $n\to \infty$, which shows that the $(f_n^{\delta})_{n\in \mathbb{N}}$ cannot be $\epsilon$-shadowed for any $\epsilon>0$. Therefore, $C_{\phi}$ does not have the positive shadowing property on $H^2(\mathbb{C}_{+}).$
\end{proof}

\begin{remark}
The main goal of this section is to classify the symbols that induce a composition operator satisfying some type of shadowing on $H^{2}(\mathbb{C}_{+})$. Proposition \ref{ksha} states that existence of fixed points for the symbol implies non-existence of positive shadowing property. This idea can be extended for composition operators on Fréchet spaces.  
\end{remark}

\begin{theorem}
Let $\phi(w)=aw+b$ be a self-map of $\mathbb{C}_{+}$ with $a>0$ and $Re (b)\geq 0.$ The following assertions hold:
\begin{enumerate}[label=S.\arabic*]
\item \label{S1} If $\phi$ is parabolic-type or a hyperbolic non-automorphism of type I, then $C_{\phi}$ does not have the positive shadowing property on $H^2(\mathbb{C}_{+})$. 
\item \label{S2} If $\phi$ is a hyperbolic automorphism or a hyperbolic non-automorphism of type II, then $C_{\phi}$ has the positive shadowing property on $H^2(\mathbb{C}_{+})$.
\end{enumerate}
\end{theorem}
\begin{proof}
\ref{S1} If $\phi$ is a parabolic map, then $a=1$ and $Re(b)=0$ or $a=1$ and $Re(b)>0.$ When $a=1$ and $Re(b)=0$, then $C_{\phi}$ is a unitary operator. In particular, $C_{\phi}$ is a normal and invertible operator. From Theorem \ref{sp1} follows that $C_{\phi}$ is not hyperbolic and it does not have the shadowing property by Corollary \ref{corsh1}. Now, if $a=1$ and $Re(b)>0,$ then $C_{\phi}$ is a normal operator (not invertible). By Theorem \ref{sp1} the spectrum $\sigma(C_{\phi})$ intersects $\mathbb{T}$. Now, we will argue as in \cite[Theorem~30]{ES}, assuming that $C_{\phi}$ has positive shadowing property. By the proof of \cite[Theorem~27]{ES}, we obtain a function $f\in H^{2}(\mathbb{C}_{+})$ such that 
\begin{equation}\label{somb1}
\frac{n\delta}{3}-1<\|C^{n}_{
\phi}f\|<n\delta +1 \ \ {\rm{for \ all}} \ n\in \mathbb{N}\cup \{0\}, 
\end{equation}
where $\delta>0$ is the constant in the definition of shadowing associated to $\epsilon=1$.
Let $S_{\phi}=C^{\ast}_{\phi}C_{\phi}$ and let $\mu$ be the spectral measure associated to $S_{\phi}$ and $f$. Then, $\mu(\sigma(S_{\phi})\cap (1,+\infty))=0,$ because $$0\leq \displaystyle\int_{\sigma(S_{\phi})}t^{n}d\mu(t)=\langle S_{\phi}^{n}f,f \rangle=\|C_{\phi}^{n}f\|^2<(n\delta +1)^{2},\ {\rm{for \ all}} \ n\in \mathbb{N}\cup \{0\}.$$ Hence, using Hölder inequality we obtain that $\|C^{n}_{\phi}f\|\leq (\mu(\sigma(S_{\phi})))^{\frac{1}{2}}=\|f\|$ for all $n\in \mathbb{N}\cup \{0\},$ which contradicts (\ref{somb1}). Consequently, $C_{\phi}$ does not have the positive shadowing property. If $\phi$ is a hyperbolic non-automorphism of type I, then $a\in (0, 1)$ and $Re(b) > 0$. Hence $\phi$ has a fixed
point in $\mathbb{C}_+$, by Proposition \ref{ksha} the operator $C_{\phi}$ does not have positive shadowing property on $H^{2}(\mathbb{C}_{+}).$ \vspace{0.2cm} 

\ref{S2} Since $\phi$ is a hyperbolic automorphism, it follows that
$a \in (0, 1) \cup (1,+\infty)$ and $Re(b) = 0$. Then $C_{\phi}$ is an invertible normal operator on $H^2(\mathbb{C}_+)$ and $\sigma(C_{\phi}) = \{z : |z| = \frac{1}{\sqrt{a}}\}$. For the case, $a\in (0, 1)$, we obtain $\sigma(C_{\phi})\cap \mathbb{T} =\emptyset$. By Corollary \ref{corsh}, $C_{\phi}$ has the shadowing property. For the case $a\in (1,+\infty)$, we have that $\sigma(C_{\phi})\subset \mathbb{D}.$ Then, there exist $p\in (0, 1)$ and $K\geq 1$ such that $\|C_{\phi}^{n}\|\leq Kp^{n}$ for all $n\in \mathbb{N}\cup \{0\}$. We argue as in \cite[Theorem~13]{ES}, for $\epsilon>0,$ we choose $\delta=\displaystyle\frac{(1-p)\epsilon}{2K}>0.$ Let $(f_{n})_{n\in \mathbb{N}\cup \{0\}}$ be a $\delta$-pseudotrajectory of $C_{\phi}$ and define the sequence $g_{n}=f_{n}-C_{\phi}f_{n-1}$. By induction, $$f_{n}=C_{\phi}^{n}f_{0}+\displaystyle\sum_{i=1}^{n} C_{\phi}^{n-i}g_{i} \  \ {\rm{ for\ all}} \ n\in \mathbb{N}.$$ 

Hence, $$\|f_{n}-C_{\phi}^{n}f_{0}\|\leq \displaystyle\sum_{i=1}^{n}\|C_{\phi}^{n-i}\|\|g_{i}\|.$$ Since $\|g_{n}\|<\delta$ for all $n\in \mathbb{N},$ it follows that $$\|f_{n}-C_{\phi}^{n}f_{0}\|\leq \displaystyle\sum_{i=1}^{n}\delta K p^{n-i}<\frac{K\delta}{1-p}=\frac{\epsilon}{2}$$ for $n\in \mathbb{N}.$ Therefore, $(f_{n})_{n\in \mathbb{N}\cup \{0\}}$ is $\epsilon$-shadowed by $(C_{\phi}^{n}f_0)_{n\in \mathbb{N}\cup \{0\}}$ and $C_{\phi}$ has the positive shadowing property. Finally, if $\phi$ is a hyperbolic non-automorphism of type II, that is, $a\in (1,+\infty)$ and $Re(b)>0.$ From Theorem \ref{sp2} follows that $\sigma (C_{\phi})=\{z: |z|\leq \frac{1}{\sqrt{a}}\}.$ Then, $\sigma(C_{\phi})\subset \mathbb{D}$ and we proceed as in the case of hyperbolic automorphism with $a\in (1,\infty)$.
\end{proof}

To combining the assertions \ref{S1} and \ref{S2} of the previous theorem, we have the following result that characterizes the positive shadowing property.

\begin{corollary}
 Let $\phi(w)=aw+b$ be a self-map of $\mathbb{C}_{+}$ with $a>0$ and $Re (b)\geq 0.$ Then, $C_{\phi}$ has positive shadowing property if and only if $\phi$ is a hyperbolic automorphism or a hyperbolic non-automorphism of type II. 
\end{corollary}

Finally we prove that $H^2(\mathbb{C}_{+})$ does not support any Li-Yorke chaotic composition operator with linear fractional symbol. This is another way to prove that the operator is not hypercyclic (see \cite[Theorem~10]{NS}).

\begin{theorem}
Let $\phi(w) = aw + b$ be a self-map of $\mathbb{C}_{+
}$ with $a > 0$ and $Re(b)\geq 0$. Then $C_{\phi}$ is not Li-Yorke chaotic on $H^{2}(\mathbb{C}_+)$.
\end{theorem}
\begin{proof}
If $a =1$ or $Re(b) = 0$, then $C_{\phi}$ is normal (Theorem \ref{normal}) and in this case $C_{\phi}$ is not Li-Yorke by Corollary \ref{mainCor}. Now we
consider the case that $\phi$ is to be a \textit{hyperbolic non-automorphism}, that is, $$\phi(w)=aw+b, \ \ \textrm{with} \ \ a\in (0,1)\cup(1,+\infty) \ \ \textrm{and} \ \ Re(b)>0.$$ 
Since $\|C_{\phi}\|= \sqrt{\frac{1}{a}}$, if $a\in (1,+\infty)$ then $\|C_{\phi}\|< 1$ which implies $\sigma(C_{\phi})\cap \mathbb{T} =\emptyset$. By Corollary \ref{mainCor}, $C_{\phi}$ is not Li-Yorke chaotic . Besides, for the case $a\in (0,1)$; suppose that $f\in H^{2} (\mathbb{C}_{+})$ nonzero then there exists $N\in\mathbb{N}$ such that $\|C_{\phi}^{N}f\|>2$ this implies that $\liminf \|C_{\phi}^{n}f\|\geq2$. Therefore, $f$ cannot be an irregular vector for $C_{\phi}$ in $H^{2} (\mathbb{C}_{+})$ and by Theorem \ref{mainLY} follows that $C_{\phi}$ is not Li-Yorke chaotic.  
\end{proof}

\section*{Acknowledgments}
We would like to express our heartfelt gratitude to the referees for the detailed and thorough review this article. CF.A. thanks ANID-Chile by the financial support to develop this project. 

\bibliography{Bibliography}
\bibliographystyle{acm}

\end{document}